\documentclass{amsart}

\usepackage{graphicx}

\newtheorem{theorem}{Theorem}[section]
\newtheorem{lemma}[theorem]{Lemma}

\theoremstyle{definition}

\theoremstyle{remark}
\newtheorem{remark}[theorem]{Remark}

\numberwithin{equation}{section}

\begin{document}

\title[DSM for solving NOE with monotone operators]
{Dynamical systems method for solving nonlinear equations
with monotone operators}


%    author one information
\author{N. S. Hoang}
\address{Mathematics Department, Kansas State University,
Manhattan, KS 66506-2602, USA}
\curraddr{}
\email{nguyenhs@math.ksu.edu}
\thanks{}

%    author two information
\author{A. G. Ramm}
\address{Mathematics Department, Kansas State University,
Manhattan, KS 66506-2602, USA}
\curraddr{}
\email{ramm@math.ksu.edu}
\thanks{}

%    \subjclass is required.

\date{}


\dedicatory{}

%    Abstract is required.
\begin{abstract}
A version of the Dynamical Systems Method (DSM) for solving ill-posed
nonlinear equations with monotone operators in a Hilbert space
is studied in this paper.
An {\it a posteriori} stopping rule, based on a discrepancy-type
principle is proposed and justified mathematically.
The results of two numerical experiments are presented. They show that
the proposed version of DSM is numerically efficient.
The numerical experiments consist of solving  nonlinear integral
equations.
\end{abstract}

\maketitle



\section{Introduction}

In this paper we study a Dynamical Systems Method (DSM) for solving the equation
\begin{equation}
\label{aeq1}
F(u)=f,
\end{equation}
 where
$F$ is a nonlinear twice Fr\'{e}chet differentiable monotone operator in a
real Hilbert space $H$, and equation \eqref{aeq1} is assumed solvable.
Monotonicity means that
\begin{equation}
\label{ceq2}
\langle F(u)-F(v),u-v\rangle\ge 0,\quad \forall u,v\in H.
\end{equation}
Here, $\langle \cdot,\cdot\rangle$ denotes the inner product in $H$.
It is known (see, e.g., \cite{R499}), that the set $\mathcal{N}:=\{u:F(u)=f\}$ is closed and convex if $F$ is monotone and continuous.
A closed and convex set in a Hilbert space has a unique minimal-norm element. This element in $\mathcal{N}$ we denote $y$, $F(y)=f$.
We assume that
\begin{equation}
\label{ceq3}
\sup_{\|u-u_0\|\le R}\|F^{(j)}(u)\|\le M_j(u_0,R),\quad 0\le j\le 2,
\end{equation}
where $u_0\in H$ is an element of $H$,
$R>0$ is arbitrary,
and $f=F(y)$ is not known; but
 $f_\delta$, the noisy data, are known and $\|f_\delta-f\|\le \delta$. If $F'(u)$
 is not boundedly invertible, then
 solving for $u$ given noisy data $f_\delta$ is often (but not always) an ill-posed problem.
When $F$ is a linear bounded operator many methods for a stable solution of \eqref{aeq1} were proposed 
(see \cite{I}--\cite{R499} and the references therein). However, when $F$
is nonlinear then the theory is
%of discrepancy type and
%the convergence of the method using such stopping time is
less complete.

The DSM for solving equation \eqref{aeq1} was extensively studied in \cite{R499}--\cite{491}.
In \cite{R499} the following version of the DSM for solving equation
\eqref{aeq1} was studied:
\begin{equation}
\label{aeq2}
\dot{u}_\delta = -\big{(}F'(u_\delta) + a(t)I\big{)}^{-1}\big{(}F(u_\delta)+a(t)u_\delta - f_\delta\big{)},
\quad u_\delta(0)=u_0.
\end{equation}
Here $F$ is a monotone operator, and $a(t)>0$ is a continuous function,
defined for all $t\geq 0$, strictly monotonically decaying,
$\lim_{t\to \infty}a(t)=0$. These assumptions on $a(t)$ hold throughout
the paper and are not repeated. Additional assumptions
on $a(t)$ will appear later.
Convergence of the above DSM  was proved in \cite{R499} for any initial
value $u_0$ with
an {\it a priori} choice of stopping time $t_\delta$,
provided that $a(t)$ is suitably chosen.

The theory of monotone operators is presented in many books, e.g.,
in \cite{D}, \cite{P}, \cite{V}.
Most of the results of the theory of monotone operators, used in
this paper, can be found in \cite{R499}.
In \cite{OR} methods for solving nonlinear equations in a
finite-dimensional space are discussed.
%In \cite{T} a
%method for solving equation \eqref{aeq1} with a monotone operator
%using a solution to a Cauchy problem is discussed. The assumptions
%in \cite{T} are much more restrictive than ours, the rule R2 on p.196

In this paper we propose and justify a stopping rule based on a discrepancy
principle (DP) for the DSM \eqref{aeq2}.
The main result of this paper is Theorem~\ref{mainthm} in which
a DP is formulated,
the existence of the stopping time $t_\delta$ is proved, and
the convergence of the DSM with the proposed DP is justified under some
natural assumptions apparently for the first time for a wide class
of nonlinear equations with monotone operators.

These results are new from the theoretical point of view and very useful
practically.
The auxiliary results in our paper are also new and can be used in other
problems of numerical analysis. These auxiliary results are formulated
in Lemmas \ref{lem0}--\ref{lem1}, \ref{lemma11}, \ref{lemauxi1}, \ref{lemauxi2}, and in the remarks. In particular, in Remark~\ref{hehehihiha} we emphasize
that the trajectory of the solution stays in a ball of a fixed radius $R$
for all $t\geq 0$.

In Section 4 the results of two numerical experiments are presented.
In the second experiment we demonstrate numerically that our method for
solving equation \eqref{aeq1} can be used even for a wider class
of equations than the basic Theorem~\ref{mainthm} guarantees.

\section{Auxiliary results}
\label{sec2}

Let us consider the following equation:
\begin{equation}
\label{2eq2}
F(V_{\delta,a})+aV_{\delta,a}-f_\delta = 0,\qquad a>0,
\end{equation}
where $a=const$. It is known (see, e.g., \cite{R499})
that equation \eqref{2eq2} with monotone continuous operator
$F$ has a unique solution for any $f_\delta\in H$.

Let us recall the following result from \cite[p. \#112]{R499}:
\begin{lemma}
\label{rejectedlem}
Assume that equation \eqref{aeq1} is solvable, $y$ is its minimal-norm solution, and assumptions
\eqref{ceq2} and \eqref{ceq3} hold. Then
$$
\lim_{a\to 0} \|V_{0,a}-y\| = 0,
$$
where $V_{0,a}$ solves \eqref{2eq2} with $\delta=0$.
\end{lemma}

%For simplicity let us denote $V_\delta:=V_{\delta,a}$ when $\delta\not=0$ and $a>0$.

\begin{lemma}
\label{lem0}
If \eqref{ceq2} holds and $F$ is continuous, then
$\|V_{\delta,a}\|=O(\frac{1}{a})$ as $a\to\infty$, and
\begin{equation}
\label{4eq2}
\lim_{a\to\infty}\|F(V_{\delta,a})-f_\delta\|=\|F(0)-f_\delta\|.
\end{equation}
\end{lemma}

\begin{proof}
%First, we claim that $aV_{\delta,a}$ is bounded for all $a>0$. Indeed,
Rewrite \eqref{2eq2} as
$$
F(V_{\delta,a}) - F(0) + aV_{\delta,a} + F(0)-f_\delta = 0.
$$
Multiply this equation
by $V_{\delta,a}$, use inequality $\langle F(V_{\delta,a})-F(0),V_{\delta,a}-0\rangle \ge 0$ from \eqref{ceq2} and get:
$$
a\|V_{\delta,a}\|^2\le\langle aV_{\delta,a} + F(V_{\delta,a})-F(0), V_{\delta,a}\rangle =
\langle f_\delta-F(0), V_{\delta,a}\rangle \le \|f_\delta-F(0)\|\|V_{\delta,a}\|.
$$
Therefore, %$\|V_{\delta,a}\|\le {\|f_\delta-F(0)\|}{a}=O{\frac 1 a}$,
%so
$\|V_{\delta,a}\|=O(\frac{1}{a})$. This and the continuity of $F$ imply \eqref{4eq2}.
\end{proof}

Let $a=a(t)$, $0<a(t)\searrow 0$, and assume $a\in C^1[0,\infty)$.
Then the solution $V_\delta(t):=V_{\delta,a(t)}$ of \eqref{2eq2} is a function of $t$.
From the triangle inequality one gets
$$
\|F(V_\delta(0))-f_\delta\|\ge\|F(0) - f_\delta\| -\|F(V_\delta(0))-F(0)\|.
$$
%One has
From Lemma~\ref{lem0} it follows that for large $a(0)$ one has
$$
\|F(V_\delta(0))-F(0)\|\le M_1\|V_\delta(0)\|=O\bigg{(}\frac{1}{a(0)}\bigg{)}.
$$
Therefore,
if $\|F(0)-f_\delta\|>C\delta$, then $\|F(V_\delta(0))-f_\delta\|\ge (C-\epsilon)\delta$,
where $\epsilon>0$ is sufficiently small, for sufficiently large $a(0)>0$.

Below the words decreasing and increasing mean strictly
decreasing and strictly increasing.

\begin{lemma}
\label{leminde}
Assume $\|F(0)-f_\delta\|>0$. Let $0<a(t)\searrow 0$, and let $F$ be monotone.
Denote
$$
\phi(t):=\|F(V_\delta(t)) - f_\delta\|,\quad  \psi(t):=\|V_\delta(t)\|,
$$
where $V_\delta(t)$ solves \eqref{2eq2} with $a=a(t)$.
Then
$\phi(t)$ is decreasing, and $\psi(t)$ is increasing.
\end{lemma}

\begin{proof}
Since $\|F(0)-f_\delta\|>0$, it follows that $\psi(t)\not=0,\, \forall t\ge0$.
Note that $\phi(t)=a(t)\|V_\delta(t)\|$. One has
\begin{equation}
\label{1eq3}
\begin{split}
0&\le \langle F(V_\delta(t_1))-F(V_\delta(t_2)),V_\delta(t_1)-V_\delta(t_2)\rangle\\
&= \langle -a(t_1)V_\delta(t_1)+a(t_2)V_\delta(t_2),V_\delta(t_1)-V_\delta(t_2)\rangle\\
&= (a(t_1)+a(t_2))\langle V_\delta(t_1),V_\delta(t_2) \rangle -a(t_1)\|V_\delta(t_1)\|^2 - a(t_2)\|V_\delta(t_2)\|^2.
\end{split}
\end{equation}
Thus,
\begin{equation}
\label{2eq6}
\begin{split}
0&\le (a(t_1)+a(t_2))\langle V_\delta(t_1),V_\delta(t_2) \rangle -a(t_1)\|V_\delta(t_1)\|^2 - a(t_2)\|V_\delta(t_2)\|^2\\
& \le  (a(t_1)+a(t_2))\|V_\delta(t_1)\|\|V_\delta(t_2) \| - a(t_1)\|V_\delta(t_1)\|^2 - a(t_2)\|V_\delta(t_2)\|^2\\
& = (a(t_1) \|V_\delta(t_1)\| - a(t_2) \|V_\delta(t_2)\|)(\|V_\delta(t_2)\|-\|V_\delta(t_1)\|)\\
& = (\phi(t_1)-\phi(t_2))(\psi(t_2)-\psi(t_1)).
\end{split}
\end{equation}
If $\psi(t_2)> \psi(t_1)$, then \eqref{2eq6} implies $\phi(t_1)\ge \phi(t_2)$, so
$$
a(t_1)\psi(t_1)\ge a(t_2)\psi(t_2)> a(t_2)\psi(t_1).
$$
Thus, if $\psi(t_2)> \psi(t_1)$, then $a(t_2)< a(t_1)$ and, therefore, $t_2> t_1$,
because $a(t)$ is decreasing.

Similarly, if $\psi(t_2)< \psi(t_1)$, then $\phi(t_1)< \phi(t_2)$.
This implies $a(t_2)> a(t_1)$, so $t_2< t_1$.

If $\psi(t_2)= \psi(t_1)$, then \eqref{1eq3} implies
$$
\|V_\delta(t_1)\|^2\le \langle V_\delta(t_1),V_\delta(t_2) \rangle \le \|V_\delta(t_1)\|\|V_\delta(t_2)\| = \|V_\delta(t_1)\|^2.
$$
This implies $V_\delta(t_1)=V_\delta(t_2)$, and then $a(t_1)=a(t_2)$. Hence, $t_1=t_2$, because $a(t)$ is decreasing.

Therefore, $\phi(t)$ is decreasing and $\psi(t)$ is increasing.
\end{proof}

\begin{lemma}
\label{lem1}
Suppose that $\|F(0)-f_\delta\|> C\delta$, \,$C>1$, and $a(0)$ is sufficiently
large. Then, there exists a unique $t_1>0$ such
that $\|F(V_\delta(t_1))-f_\delta\|=C\delta$.
\end{lemma}

\begin{proof}
The uniqueness of $t_1$ follows from Lemma~\ref{leminde}.
We have $F(y)=f$, and
\begin{align*}
0 = &\langle F(V_\delta)+aV_\delta-f_\delta, F(V_\delta)-f_\delta \rangle\\
 = &\|F(V_\delta)-f_\delta\|^2+a\langle V_\delta-y, F(V_\delta)-f_\delta \rangle + a\langle y, F(V_\delta)-f_\delta \rangle\\
 = &\|F(V_\delta)-f_\delta\|^2+a\langle V_\delta-y, F(V_\delta)-F(y) \rangle + a\langle V_\delta-y, f-f_\delta \rangle \\
 &+ a\langle y, F(V_\delta)-f_\delta \rangle\\
\ge & \|F(V_\delta)-f_\delta\|^2 + a\langle V_\delta-y, f-f_\delta \rangle + a\langle y, F(V_\delta)-f_\delta \rangle.
\end{align*}
Here the inequality $\langle V_\delta-y, F(V_\delta)-F(y) \rangle\ge0$ was used.
Therefore,
\begin{equation}
\label{1eq1}
\begin{split}
\|F(V_\delta)-f_\delta\|^2 &\le -a\langle V_\delta-y, f-f_\delta \rangle - a\langle y, F(V_\delta)-f_\delta \rangle\\
&\le a\|V_\delta-y\| \|f-f_\delta\| + a\|y\| \|F(V_\delta)-f_\delta\|\\
&\le  a\delta \|V_\delta-y\|  + a\|y\| \|F(V_\delta)-f_\delta\|.
\end{split}
\end{equation}
On the other hand, we have
\begin{align*}
0&= \langle F(V_\delta)-F(y) + aV_\delta +f -f_\delta, V_\delta-y\rangle\\
&=\langle F(V_\delta)-F(y),V_\delta-y\rangle + a\| V_\delta-y\| ^2 + a\langle y, V_\delta-y\rangle + \langle f-f_\delta, V_\delta-y\rangle\\
&\ge  a\| V_\delta-y\| ^2 + a\langle y, V_\delta-y\rangle + \langle f-f_\delta, V_\delta-y\rangle,
\end{align*}
where the inequality $\langle V_\delta-y, F(V_\delta)-F(y) \rangle\ge0$ was used. Therefore,
$$
a\|V_\delta-y\|^2 \le a\|y\|\|V_\delta-y\|+\delta\|V_\delta-y\|.
$$
This implies
\begin{equation}
\label{1eq2}
a\|V_\delta-y\|\le a\|y\|+\delta.
\end{equation}
From \eqref{1eq1} and \eqref{1eq2}, and an elementary inequality $ab\le \epsilon a^2+\frac{b^2}{4\epsilon},\,\forall\epsilon>0$, one gets
\begin{equation}
\label{3eq4}
\begin{split}
\|F(V_\delta)-f_\delta\|^2&\le \delta^2 + a\|y\|\delta + a\|y\| \|F(V_\delta)-f_\delta\|\\
&\le \delta^2 + a\|y\|\delta + \epsilon \|F(V_\delta)-f_\delta\|^2 +
\frac{1}{4\epsilon}a^2\|y\|^2,
\end{split}
\end{equation}
where $\epsilon>0$ is fixed, independent of $t$, and can be chosen
arbitrarily small.
Let $t\to\infty$ and $a=a(t)\searrow 0$. Then \eqref{3eq4} implies
$\limsup_{t\to\infty}(1-\epsilon)\|F(V_\delta)-f_\delta\|^2\le \delta^2$.
This, the continuity of $F$, the continuity of $V_\delta(t)$ on $[0,\infty)$,
and the assumption $\|F(0)-f_\delta\|>C\delta$, where $C>1$,
imply that equation $\|F(V_\delta(t))-f_\delta\|=C\delta$ must have a
solution $t_1>0$.
\end{proof}

\begin{remark}
{\rm
Let $V:=V_\delta(t)|_{\delta=0}$, so $F(V)+a(t)V-f=0$.
Let $y$ be the minimal-norm solution to $F(u)=f$.
We claim that
\begin{equation}
\label{rejected11}
\|V_{\delta}-V\|\le \frac{\delta}{a}.
\end{equation}
Indeed, from \eqref{2eq2} one gets
$$
F(V_{\delta}) - F(V) + a (V_{\delta}-V)=f- f_\delta.
$$
Multiply this equality by $(V_{\delta}-V)$ and use \eqref{ceq2} to obtain
\begin{align*}
\delta \|V_{\delta}-V\| &\ge \langle f-f_\delta, V_{\delta}-V \rangle\\
&= \langle F(V_{\delta}) - F(V) + a (V_{\delta}-V), V_{\delta}-V \rangle\\
&\ge a \|V_{\delta}-V\|^2.
\end{align*}
This implies \eqref{rejected11}.

Similarly, from the equation
$$
F(V) + a V -F(y)=0,
$$
one can derive that
\begin{equation}
\label{rejected12}
\|V\| \le \|y\|.
\end{equation}
%%Similar arguments one can find in \cite{R499}.
%
From \eqref{rejected11} and \eqref{rejected12}, one gets the following estimate:
\begin{equation}
\label{2eq1}
\|V_\delta\|\le \|V\|+\frac{\delta}{a}\le \|y\|+\frac{\delta}{a}.
\end{equation}
}
\end{remark}

%%%%%%%%%%%%%%%%%%%%%%%%%%%%%%%%%%%%%%%%%%%%%%

Let us recall the following lemma, which is basic in our proofs.
\begin{lemma}[\cite{R499}, p. 97]
\label{lemramm}
Let $\alpha(t),\,\beta(t),\,\gamma(t)$ be continuous
nonnegative functions on $[\tau_0,\infty)$,\, $\tau_0\ge 0$
is a fixed number. If there exists a function $\mu:=\mu(t)$, 
$$
\mu\in C^1[\tau_0,\infty),\quad \mu(t)>0, \quad \lim_{t\to\infty} \mu(t)=\infty,
$$
such that
\begin{align}
\label{1eq4}
0\le \alpha(t)&\le \frac{\mu(t)}{2}\bigg{[}\gamma -\frac{\dot{\mu}(t)}{\mu(t)}\bigg{]},
\qquad \dot{u}:=\frac{du}{dt},\\
\label{1eq5}
\beta(t)      &\le \frac{1}{2\mu(t)}\bigg{[}\gamma -\frac{\dot{\mu}(t)}{\mu(t)}\bigg{]},\\
\label{1eq6}
\mu(\tau_0)g(\tau_0)    &< 1,
\end{align}
and $g(t)\ge 0$ satisfies the inequality
\begin{equation}
\label{1eq7}
\dot{g}(t)\le -\gamma(t)g(t)+\alpha(t)g^2(t)+\beta(t),\quad t\ge \tau_0,
\end{equation}
then
%$g(t)$ exists on $[\tau_0,\infty)$ and
\begin{equation}
\label{3eq10}
0\le g(t) < \frac{1}{\mu(t)}\to 0,\quad \text{as} \quad t\to\infty.
\end{equation}
If inequalities \eqref{1eq4}--\eqref{1eq6} hold on an interval $[\tau_0,T)$, then
$g(t)$, the solution to inequality \eqref{1eq7}, exists on this interval and inequality \eqref{3eq10} holds on $[\tau_0,T)$.
\end{lemma}

\begin{lemma}
\label{lemma11}
Suppose $M_1, c_0$, and $c_1$ are positive constants and $0\not=y\in H$.
Then there exist $\lambda>0$ and a function $a(t)\in C^1[0,\infty)$, $0<a(t)\searrow 0$, such that
%$$
%|\dot{a}(t)|\le \frac{a^3(t)}{4},
%$$
%and
the following conditions hold:
\begin{align}
\label{eqzx0}
\frac{M_1}{\|y\|}&\le \lambda,\\
\label{eqzx1}
\frac{c_0}{a(t)}&\le \frac{\lambda}{2a(t)}\bigg{[}1-\frac{|\dot{a}(t)|}{a(t)}\bigg{]},\\
\label{eqzx2}
c_1\frac{|\dot{a}(t)|}{a(t)}&\le \frac{a(t)}{2\lambda}\bigg{[}1-\frac{|\dot{a}(t)|}{a(t)}\bigg{]},\\
\label{eqzx3}
\|F(0) - f_\delta\|& \le \frac{a^2(0)}{\lambda}.
\end{align}
\end{lemma}

\begin{proof}
Take
\begin{equation}
\label{24eq3}
a(t) = \frac{d}{(c+t)^b},\quad 0<b\le 1,\quad c\ge \max \big{(}2b,1\big{)}.
\end{equation}
Note that $|\dot{a}|=-\dot{a}$.
We have
\begin{equation}
\label{eql12}
\frac{|\dot{a}(t)|}{a(t)}=\frac{b}{c+t}\le
\frac{b}{c}\le\frac{1}{2},\qquad \forall t\ge 0.
\end{equation}
Hence,
\begin{equation}
\label{24eq2}
\frac{1}{2}\le 1-\frac{|\dot{a}(t)|}{a(t)},\qquad \forall t\ge 0 .
\end{equation}
Take
\begin{equation}
\label{3eq19}
\lambda\ge \frac{M_1}{\|y\|}.
\end{equation}
Then \eqref{eqzx0} is satisfied.

Choose $d$ such that
\begin{equation}
\label{eq11-5}
d \ge \max \bigg{(}\sqrt{c^{2b}\lambda \|F(0) - f_\delta\|}, 4b\lambda c_1 \bigg{)}.
\end{equation}
From equality \eqref{24eq3} and inequality \eqref{eq11-5} one gets
\begin{equation}
\frac{|\dot{a}(t)|}{a^2(t)} = \frac{b}{d(c+t)^{1-b}} \le \frac{b}{d} \le \frac{1}{4\lambda c_1},\qquad \forall t\ge 0.
\end{equation}
This and inequality \eqref{eql12} imply inequality \eqref{eqzx2}.
It follows from inequality \eqref{eq11-5} that
\begin{equation}
\|F(0) - f_\delta\| \le
\frac{d^2}{c^{2b}\lambda}=\frac{a^2(0)}{\lambda}.
\end{equation}
Thus, inequality \eqref{eqzx3} is satisfied.

Choose $\kappa\ge1$ such that
\begin{equation}
\label{23eq3}
\kappa > \max\bigg{(}\frac{4c_0}{\lambda}, 1\bigg{)}.
\end{equation}
Define
\begin{equation}
\label{23eq4}
\nu(t):=\kappa a(t),\quad \lambda_\kappa := \kappa \lambda.
\end{equation}
Note that inequalities \eqref{eqzx0}, \eqref{eqzx2}, \eqref{eqzx3}
and \eqref{eql12} still hold for $a(t)=\nu(t)$ and $\lambda = \lambda_\kappa$.

Using the inequalities \eqref{23eq3} and $c\ge 1$ and the definition
\eqref{23eq4}, one obtains
\begin{equation}
\frac{c_0}{\nu (t)}
\le \frac{\lambda \kappa }{4\nu(t)}
\le \frac{\lambda_\kappa}{2\nu(t)} \bigg{[}1-\frac{|\dot{\nu }|}{\nu }
\bigg{]}.
\end{equation}
Thus, one can replace the function $a(t)$ by $\nu(t)=\kappa a(t)$
and $\lambda$ by $\lambda=\lambda_\kappa$
to satisfy inequalities \eqref{eqzx0}--\eqref{eqzx3}.
\end{proof}

\begin{remark}
\label{rem8}
{\rm
In the proof of Lemma~\ref{lemma11} $a(0)$ and $\lambda$ can be chosen so
that $\frac{a(0)}{\lambda}$
is uniformly bounded as $\delta \to 0$ regardless of the rate
of growth of the constant $M_1=M_1(R)$ from formula \eqref{ceq3} when $R\to\infty$,
i.e., regardless of the strength of the nonlinearity $F(u)$.

Indeed, to satisfy \eqref{3eq19} one can choose $\lambda = \frac{M_1}{\|y\|}$.
To satisfy \eqref{eq11-5} one can choose
$$
d = \max \bigg{(}\sqrt{c^{2b}\lambda \|f_\delta - F(0)\|} ,
4b\lambda c_1\bigg{)}\leq \max \bigg{(}\sqrt{c^{2b}\lambda (\|f-F(0)\|+1)} ,
4b\lambda c_1\bigg{)},
$$
where we have assumed, without loss of generality, that $0<\delta< 1$.
With this choice of $d$ and $\lambda$, the ratio $\frac {a(0)}{\lambda}$
is bounded uniformly with respect to $\delta\in (0,1)$ and does not
depend on $R$.

Indeed, with the above choice one has $\frac {a(0)}{\lambda}=\frac {d}{c^{b}\lambda}\leq
\tilde{c}(1+\sqrt{\lambda^{-1}})\leq \tilde{c}$,
where $\tilde{c}>0$ is a constant independent of $\delta$, and one can
assume that $\lambda\geq 1$ without loss of generality.

This remark is used in Remark~\ref{hehehihiha}, where we prove
that the trajectory of $u_\delta(t)$, defined by \eqref{3eq12}, stays in a
ball $B(u_0,R)$ for all $0\le t\le t_\delta$, where the number $t_\delta$
is defined by formula \eqref{2eq3} (see below), and $R>0$ is sufficiently large.
An upper bound on $R$ is given in Remark~\ref{hehehihiha}.
}
\end{remark}

\begin{remark}
\label{xrem2}
{\rm
It is easy to choose $u_0\in H$ such that
\begin{equation}
\label{teq20}
g_0:=\|u_0-V_\delta(0) \|\le \frac{\|F(0)-f_\delta\|}{a(0)}.
\end{equation}
Indeed, if, for example, $u_0=0$, then by Lemmas~\ref{lem0} and \ref{leminde} one gets
$$
g_0=\|V_\delta(0)\|=\frac{a(0)\|V_\delta(0)\|}{a(0)} \le \frac{\|F(0)-f_\delta\|}{a(0)}.
$$
If \eqref{eqzx3} and \eqref{teq20} hold, then
$g_0 \le \frac{a(0)}{\lambda}.$ Inequality \eqref{teq20}
also holds if $||u_0-V_\delta(0)||$ is sufficiently small.
}
\end{remark}

\begin{lemma}
\label{lemauxi1}
Let $p,b$ and $c$ be positive constants. Then
\begin{equation}
\label{auxi1}
\bigg{(}p-\frac{b}{c}\bigg{)}\int_0^t \frac{e^{ps}}{(s+c)^b} ds < \frac{e^{pt}}{(c+t)^b},\qquad \forall c,b> 0,\quad t>0.
\end{equation}
\end{lemma}

\begin{proof}
One has
\begin{align*}
\frac{d}{dt}\bigg{(}\frac{e^{pt}}{(c+t)^b}\bigg{)} &= \frac{pe^{pt}}{(c+t)^b}
- \frac{be^{pt}}{(c+t)^{b+1}}\\
&\ge \bigg{(}p-\frac{b}{c}\bigg{)}\frac{e^{pt}}{(c+t)^b},\qquad t\ge 0.
\end{align*}
Therefore,
\begin{align*}
\bigg{(}p-\frac{b}{c}\bigg{)}\int_0^t \frac{e^{ps}}{(s+c)^b} ds
&\le \int_0^t\frac{d}{ds}\frac{e^{ps}}{(c+s)^b}ds\\
&\le \frac{e^{pt}}{(c+t)^b} - \frac{1}{c^b}\le \frac{e^{pt}}{(c+t)^b}.
\end{align*}
Lemma~\ref{lemauxi1} is proved.
\end{proof}

\begin{lemma}
\label{lemauxi2}
Let $a(t)=\frac{d}{(c+t)^b}$ where $d,c,b>0$,\,
$c\ge 6b$. One has
\begin{equation}
\label{auxieq2}
e^{-\frac{t}{2}}\int_0^t e^\frac{s}{2}|\dot{a}(s)|\|V_\delta(s)\|ds \le
 \frac{1}{2}a(t)\|V_\delta(t)\|,\qquad t\ge 0.
\end{equation}
\end{lemma}

\begin{proof}
Let $p=\frac{1}{2}$ in Lemma \ref{lemauxi1}. Then
\begin{equation}
\label{yeq54}
\bigg{(}\frac{1}{2}-\frac{b}{c}\bigg{)}\int_0^t \frac{e^{\frac{s}{2}}}{(s+c)^b} ds < \frac{e^{\frac{t}{2}}}{(c+t)^b},\qquad \forall c,b\ge 0.
\end{equation}
Since $c\ge 6b$ or $\frac{3b}{c}\le \frac{1}{2}$, one has
$$
\frac{1}{2} - \frac{b}{c} \ge \frac{2b}{c} \ge \frac{2b}{c+s},\qquad  s\ge 0.
$$
This implies
\begin{equation}
\label{yeq55}
a(s)\bigg{(}\frac{1}{2} -\frac{b}{c}\bigg{)}
= \frac{d}{(c+s)^b}\bigg{(}\frac{1}{2} -\frac{b}{c}\bigg{)} \ge \frac{2db}{(c+s)^{b+1}}=2|\dot{a}(s)|,\qquad s\ge 0.
\end{equation}
Multiplying \eqref{yeq55} by $e^{\frac{s}{2}}\|V_\delta(s)\|$,
integrating from $0$ to $t$, using
inequality \eqref{yeq54} and the fact that $\|V_\delta(s)\|$
is nondecreasing, one gets
$$
 e^{\frac{t}{2}}a(t)\|V_\delta(t)\| > \int_0^t e^{\frac{s}{2}}\|V_\delta(t)\|a(s)\bigg{(}\frac{1}{2}-\frac{b}{c}\bigg{)} ds
 \ge 2\int_0^t e^{\frac{s}{2}} |\dot{a}(s)| \|V_\delta(s)\|  ds,\qquad t\ge 0.
$$
This implies inequality \eqref{auxieq2}. Lemma~\ref{lemauxi2} is proved.
\end{proof}

\section{Main result}
\label{mainsec}

%Assume:
%\begin{equation}
%\label{3eq11}
%0<a(t)\searrow 0,\quad \lim_{t\to\infty}\frac{\dot{a}}{a}=0,
%\quad \frac{|\dot{a}|}{a}\le \frac{1}{2}
%,\quad \frac{|\dot{a}|}{a^2}\le 1.
%\end{equation}
%There are many $a(t)$ satisfying these conditions.
%For example, one may take $a(t)=\frac {d}{(c+t)^b}$,
%where $c,d,$ and $b>0$ are constants, $2b\leq c$,
%$0<b<1$, and $b<d c^{1-b}$.

Denote
$$
A:=F'(u_\delta(t)),\quad A_a:=A + aI,
$$
where $I$ is the identity operator,
and $u_\delta(t)$ solves the following Cauchy problem:
\begin{equation}
\label{3eq12}
\dot{u}_\delta = -A_{a(t)}^{-1}[F(u_\delta)+a(t)u_\delta-f_\delta],
\quad u_\delta(0)=u_0.
\end{equation}
We assume below that $||F(u_0)-f_\delta||>C_1\delta^{\zeta}$,
where $C_1>1$ and $\zeta\in (0,1]$ are some constants.
We also assume, without loss of generality, that $\delta\in (0,1)$.
%Indeed, if $\delta>1$, then one can multiply \eqref{aeq1}
%by a suitable constant and transform equation \eqref{aeq1}
%into the equation in which the noise level $c\delta<1$. Under such
%transformation the problem \eqref{3eq12} remains invariant because
%the constant will be cancelled.

Assume that equation $F(u)=f$ has a solution, possibly nonunique,
and $y$ is the minimal norm solution to this equation.
Let $f$ be unknown, but $f_\delta$ be given, $\|f_\delta-f\|\le \delta$.

\begin{theorem}
\label{mainthm}
Assume
%$a(t)$
%satisfy conditions \eqref{3eq11}, and choose
$a(t)=\frac{d}{(c+t)^b}$,
where $b\in(0,1]$,\,
$c,d>0$ are constants, $c>6b$,
and $d$ is sufficiently large so that conditions \eqref{eqzx1}--\eqref{eqzx3} hold.
Assume that $F:H\to H$ is a monotone operator, twice Fr\'{e}chet
differentiable, $\sup_{u\in B(u_0,R)}\|F^{(j)}(u)\|\le M_j(u_0,R),\, 0\le j\le 2$,
$B(u_0,R):=\{u:\|u-u_0\|\le R\}$, $u_0$ is an element of $H$, satisfying inequality
\eqref{teq20} and
%\eqref{deq47} (see below).
\begin{equation}
\label{eqx47}
\|F(u_0) + a(0)u_0 -f_\delta\|\le \frac{1}{4}a(0)\|V_\delta(0)\|,
\end{equation}
where $V_\delta(t):=V_{\delta,a(t)}$ solves \eqref{2eq2} with $a=a(t)$.  
Then the solution $u_\delta(t)$ to problem \eqref{3eq12}
exists on an interval $[0,T_\delta]$,\, $\lim_{\delta\to0}T_\delta=\infty$, and
there exists a unique $t_\delta$, $t_\delta\in (0,T_\delta)$ such that
$\lim_{\delta\to 0}t_\delta=\infty$ and
\begin{equation}
\label{2eq3}
\|F(u_\delta(t_\delta))-f_\delta\|=C_1\delta^\zeta,
\quad \|F(u_\delta(t))-f_\delta\|> C_1\delta^\zeta, \quad \forall t \in [0,t_\delta),
\end{equation}
where $C_1>1$ and $0<\zeta\le 1$. If $\zeta\in (0,1)$ and $t_\delta$
satisfies \eqref{2eq3}, then
\begin{equation}
\label{2eq4}
\lim_{\delta\to 0} \|u_\delta(t_\delta) - y\|=0.
\end{equation}
\end{theorem}

\begin{remark}
{\rm
One can choose $u_0$ satisfying  inequalities \eqref{teq20} and \eqref{eqx47} (see also \eqref{deq47} below).
Indeed, if $u_0$ is a sufficiently close approximation to $V_\delta(0)$, the solution to
equation \eqref{2eq2}, then inequalities \eqref{teq20} and \eqref{eqx47} are satisfied.
Note that inequality \eqref{eqx47} is a sufficient condition for
\eqref{1eq20} to hold. In our proof inequality \eqref{1eq20}
is used at $t=t_\delta$.
The stopping time $t_\delta$ is often sufficiently large for
the quantity $e^{-\frac{t_\delta}{2}}h_0$ to be  small. In this case
inequality \eqref{1eq20} with $t=t_\delta$ is satisfied for a wide range of
$u_0$.
For example,
in our numerical experiment in Section 4
the method converged rapidly when $u_0=0$. 
Condition $c>6b$ is used in the proof of Lemma \ref{lemauxi2}.
}
\end{remark}

\begin{proof}[Proof of Theorem \ref{mainthm}]
Denote
\begin{equation}
\label{beq18}
C:=\frac{C_1+1}{2}.
\end{equation}
Let
$$
w:=u_\delta-V_\delta,\quad g(t):=\|w\|.
$$
One has
\begin{equation}
\label{1eq8}
\dot{w}=-\dot{V}_\delta-A_{a(t)}^{-1}\big{[}F(u_\delta)-F(V_\delta)+a(t)w\big{]}.
\end{equation}
We use Taylor's formula and get
\begin{equation}
\label{1eq9}
F(u_\delta)-F(V_\delta)+aw=A_a w+ K, \quad \|K\| \le\frac{M_2}{2}\|w\|^2,
\end{equation}
where $K:=F(u_\delta)-F(V_\delta)-Aw$, and $M_2$ is the constant from the estimate \eqref{ceq3}.
%$$
%\sup_{v\in B(v_0,R)} \|F''(v)\|\le M_2(R)=:M_2.
%$$
Multiplying \eqref{1eq8} by $w$ and using \eqref{1eq9} one gets
\begin{equation}
\label{3eq17}
g\dot{g}\le -g^2+\frac{M_2}{2}\|A_{a(t)}^{-1}\|g^3+\|\dot{V}_\delta\|g.
\end{equation}
Let $t_0$ be such that
\begin{equation}
\label{4eq18}
\frac{\delta}{a(t_0)}= \frac{1}{C-1}\|y\|,\qquad C>1.
\end{equation}
This $t_0$ exists and is unique since $a(t)>0$ monotonically decays
to 0 as $t\to\infty$.

Since $a(t)>0$  monotonically decays, one has
\begin{equation}
\label{eqthieu}
\frac{\delta}{a(t)}\le\frac{1}{C-1}\|y\|,\qquad 0\leq t\leq t_0.
\end{equation}
By Lemma~\ref{lem1},
there exists $t_1$ such that
\begin{equation}
\label{3eq18}
\|F(V_\delta(t_1))-f_\delta\|=C\delta,\quad F(V_\delta(t_1))+a(t_1)V_\delta(t_1)-f_\delta=0.
\end{equation}
{\it We claim that $t_1\in[0,t_0]$.}

 Indeed, from \eqref{2eq2} and
\eqref{2eq1} one gets
$$
C\delta=a(t_1)\|V_\delta(t_1)\|\le a(t_1)\bigg{(}\|y\|+ \frac{\delta}{a(t_1)}\bigg{)}=a(t_1)\|y\|+\delta,\quad C>1,
$$
so
$$
\delta\le \frac{a(t_1)\|y\|}{C-1}.
$$
Thus,
$$
\frac{\delta}{a(t_1)}\le \frac{\|y\|}{C-1}=\frac{\delta}{a(t_0)}.
$$
Since $a(t)\searrow 0$, one has $t_1\le t_0$.

Differentiating both sides of \eqref{2eq2} with respect to $t$, one obtains
$$
A_{a(t)}\dot{V_\delta} = -\dot{a}V_\delta.
$$
This implies
\begin{equation}
\label{beq24}
\begin{split}
\|\dot{V_\delta}\|\le |\dot{a}|\|A_{a(t)}^{-1} V_\delta\| 
&\le \frac{|\dot{a}|}{a}\|V_\delta\|\le \frac{|\dot{a}|}{a}\bigg{(}\|y\|+\frac{\delta}{a}\bigg{)}\\ 
&\le \frac{|\dot{a}|}{a}\|y\|\bigg{(}1+\frac{1}{C-1}\bigg{)},\qquad \qquad \forall t\le t_0.
\end{split}
\end{equation}
Since $g\ge 0$, inequalities \eqref{3eq17} and \eqref{beq24} imply
\begin{equation}
\label{1eq10}
\dot{g}\le -g(t)+\frac{c_0}{a(t)}g^2+\frac{|\dot{a}|}{a(t)}c_1,
\quad c_0=\frac{M_2}{2},\quad c_1=\|y\|\bigg{(}1+\frac{1}{C-1}\bigg{)}.
\end{equation}
Here we have used the estimate
$$
\|A_a^{-1}\|\le \frac{1}{a}
$$
and the relations
$$
A_a:=F'(u)+aI,\quad F'(u):=A\ge 0.
$$

Inequality \eqref{1eq10} is of the type \eqref{1eq7} with
$$
\gamma(t)=1,\quad \alpha(t)=\frac{c_0}{a(t)},\quad \beta(t)=c_1\frac{|\dot{a}|}{a(t)}.
$$
Let us check assumptions \eqref{1eq4}--\eqref{1eq6}.
Take
$$
\mu(t)=\frac{\lambda}{ a(t)},
$$
where $\lambda=const>0$ and satisfies conditions \eqref{1eq4}--\eqref{1eq6} in Lemma~\ref{lemma11}.
%By Lemma~\ref{lemma11} there exist $\lambda$ and $a(t)$ such that conditions
%\eqref{1eq4}--\eqref{1eq6} hold.
Since $u_0$ satisfies inequality \eqref{teq20}, one gets $g(0)\le \frac{a(0)}{\lambda}$, by Remark~\ref{xrem2}.
This, inequalities \eqref{1eq4}--\eqref{1eq6}, and Lemma \ref{lemramm}
yield
\begin{equation}
\label{eq*}
g(t)<\frac{a(t)}{\lambda},\quad \forall t\le t_0, \qquad
g(t):=\|u_\delta(t)-V_\delta(t)\|.
\end{equation}
Therefore,
\begin{equation}
\label{1eq11}
\begin{split}
\|F(u_\delta(t))-f_\delta\|\le& \|F(u_\delta(t))-F(V_\delta(t))\|+\|F(V_\delta(t))-f_\delta\|\\
\le& M_1g(t)+\|F(V_\delta(t))-f_\delta\|\\
\le& \frac{M_1a(t)}{\lambda} + \|F(V_\delta(t))-f_\delta\|,\qquad \forall t\le t_0.
\end{split}
\end{equation}
It is proved in Section~\ref{sec2}, Lemma~\ref{leminde},
that $\|F(V_\delta(t))-f_\delta\|$ {\it is decreasing}.
Since $t_1\le t_0$ , one gets
\begin{equation}
\label{3eq21}
\|F(V_\delta(t_0))-f_\delta\|\le \|F(V_\delta(t_1))-f_\delta\|= C\delta.
\end{equation}
This, inequality \eqref{1eq11}, the inequality $\frac{M_1}{\lambda}\le \|y\|$ (see \eqref{3eq19}), the relation \eqref{4eq18},
 and the definition $C_1=2C-1$ (see \eqref{beq18}), imply
\begin{equation}
\label{1eq12}
\begin{split}
\|F(u_\delta(t_0))-f_\delta\|
\le& \frac{M_1a(t_0)}{\lambda} + C\delta\\
\le& \frac{M_1\delta (C-1)}{\lambda\|y\|} + C\delta\le (2C-1)\delta=C_1\delta.
\end{split}
\end{equation}
%We have used the fact that $\frac{M_1}{\lambda}\le \|y\|$.
Thus, if
$$
\|F(u_\delta(0))-f_\delta\|> C_1\delta^\gamma,\quad 0<\gamma\le 1,
$$
then, by the continuity of the function $t\to
\|F(u_\delta(t))-f_\delta\|$ on $[0,\infty)$,
there exists $t_\delta \in (0,t_0)$ such that
\begin{equation}
\label{3eq23}
\|F(u_\delta(t_\delta))-f_\delta\|=C_1\delta^\gamma
\end{equation}
for any given $\gamma\in (0,1]$, and any fixed $C_1>1$.

{\it Let us prove \eqref{2eq4}}.

From \eqref{1eq11}  with $t=t_\delta$,
and from \eqref{2eq1}, one gets
\begin{align*}
C_1\delta^\zeta &\le M_1 \frac{a(t_\delta)}{\lambda} +
a(t_\delta)\|V_\delta(t_\delta)\|\\
&\le M_1 \frac{a(t_\delta)}{\lambda} + \|y\|a(t_\delta)+\delta.
\end{align*}
Thus, for sufficiently small $\delta$, one gets
$$
\tilde{C}\delta^\zeta \le a(t_\delta)
\bigg{(}\frac{M_1}{\lambda}+\|y\|\bigg{)},\quad \tilde{C}>0,
$$
where $\tilde{C}< C_1$ is a constant.
Therefore,
\begin{equation}
\label{1eq14}
\lim_{\delta\to 0} \frac{\delta}{a(t_\delta)}\le
\lim_{\delta\to 0}
\frac{\delta^{1-\zeta}}{\tilde{C}}\bigg{(}\frac{M_1}{\lambda}+\|y\|\bigg{)}
=0,\quad 0<\zeta<1.
\end{equation}

{\it We claim that}
\begin{equation}
\label{1eq15}
\lim_{\delta\to0}t_\delta = \infty.
\end{equation}
Let us prove \eqref{1eq15}.
Using \eqref{3eq12}, one obtains
$$
\frac{d}{dt}\big{(}F(u_\delta)+au_\delta - f_\delta\big{)}= A_a\dot{u}_\delta + \dot{a}u_\delta
= -\big{(}F(u_\delta)+au_\delta - f_\delta\big{)} + \dot{a}u_\delta.
$$
This and \eqref{2eq2} imply
\begin{equation}
\label{beq32}
\frac{d}{dt}\big{[}F(u_\delta)-F(V_\delta)+a(u_\delta-V_\delta)\big{]}
= -\big{[}F(u_\delta)-F(V_\delta) + a(u_\delta - V_\delta)\big{]} + \dot{a}u_\delta.
\end{equation}
Denote
$$
v:=v(t):=F(u_\delta(t))-F(V_\delta(t))+a(t)(u_\delta(t)-V_\delta(t)),\qquad
h:=h(t):=\|v\|.
$$
Multiplying \eqref{beq32} by $v$, one obtains
\begin{equation}
\label{2eq5}
\begin{split}
h\dot{h} &= -h^2 +\langle v,\dot{a}(u_\delta-V_\delta)\rangle + \dot{a}
\langle v,V_\delta\rangle\\
&\le -h^2 + h|\dot{a}|\|u_\delta-V_\delta\| + |\dot{a}|h\|V_\delta\|,
\qquad h\ge 0.
\end{split}
\end{equation}
Thus,
\begin{equation}
\label{beq34}
\dot{h}\le -h + |\dot{a}|\|u_\delta - V_\delta\| + |\dot{a}|\|V_\delta\|.
\end{equation}
Since $\langle F(u_\delta)-F(V_\delta),u_\delta-V_\delta\rangle \ge 0$, one
obtains from the two equations
$$
\langle v, u_\delta-V_\delta\rangle=
\langle F(u_\delta)-F(V_\delta) +a(t)(u_\delta-V_\delta),u_\delta-V_\delta
\rangle
$$
and
$$
\langle v,F(u_\delta)-F(V_\delta)\rangle=\|F(u_\delta)-F(V_\delta)\|^2
+a(t)
\langle u_\delta-V_\delta, F(u_\delta)-F(V_\delta) \rangle,
$$ 
the
following
two inequalities:
\begin{equation}
\label{beq35}
a\|u_\delta - V_\delta\|^2 \le \langle v, u_\delta-V_\delta \rangle \le
\|u_\delta - V_\delta\|h
\end{equation}
and
\begin{equation}
\label{beq36}
\|F(u_\delta)-F(V_\delta)\|^2\le \langle v, F(u_\delta)-F(V_\delta) \rangle
\le h\|F(u_\delta)-F(V_\delta)\|.
\end{equation}
Inequalities \eqref{beq35} and \eqref{beq36} imply
\begin{equation}
\label{1eq16}
a\|u_\delta-V_\delta\|\le h,\quad \|F(u_\delta)-F(V_\delta)\|\le h.
\end{equation}
Inequalities \eqref{beq34} and \eqref{1eq16} imply
\begin{equation}
\label{beq38}
\dot{h} \le -h\bigg{(}1-\frac{|\dot{a}|}{a}\bigg{)} +|\dot{a}|\|V_\delta\|.
\end{equation}
Since $1-\frac{|\dot{a}|}{a}\ge \frac{1}{2}$ because $c\ge 2b$, inequality \eqref{beq38} holds if
\begin{equation}
\label{beq39}
\dot{h} \le -\frac{1}{2}h + |\dot{a}|\|V_\delta\|.
\end{equation}
Inequality \eqref{beq39} implies
\begin{equation}
\label{1eq17}
h(t)\le h(0)e^{-\frac{t}{2}} + e^{-\frac{t}{2}}\int_0^t e^{\frac{s}{2}}|\dot{a}|\|V_\delta\|ds.
\end{equation}
From \eqref{1eq17} and \eqref{1eq16}, one gets
\begin{equation}
\|F(u_\delta(t))-F(V_\delta(t))\| \le
h(0)e^{-\frac{t}{2}} + e^{-\frac{t}{2}}\int_0^t e^{\frac{s}{2}}|\dot{a}|\|V_\delta\|ds.
\end{equation}
Therefore,
\begin{equation}
\label{1eq18}
\begin{split}
\|F(u_\delta(t))-f_\delta\|&\ge \|F(V_\delta(t))-f_\delta\|-\|F(V_\delta(t))-F(u_\delta(t))\|\\
&\ge a(t)\|V_\delta(t)\| - h(0)e^{-\frac{t}{2}} -
e^{-\frac{t}{2}}\int_0^t e^{\frac{s}{2}} |\dot{a}| \|V_\delta\|ds.
\end{split}
\end{equation}
From the results in Section~\ref{sec2}  (see Lemma~\ref{lemauxi2}), it follows that
there exists an
$a(t)$ such that
\begin{equation}
\label{1eq19}
\frac{1}{2}a(t)\|V_\delta(t)\| \ge  e^{-\frac{t}{2}}\int_0^te^\frac{s}{2}|\dot{a}| \|V_\delta(s)\|ds.
\end{equation}
For example, one can choose
\begin{equation}
\label{ddeq47}
a(t)=\frac{d}{(c+t)^b}, \quad 6b<c,
\end{equation}
where
$d,c,b>0$.
Moreover, one can always choose $u_0$ such that
\begin{equation}
\label{deq47}
h(0)=\|F(u_0) + a(0)u_0 -f_\delta\| \le \frac{1}{4} a(0)\|V_\delta(0)\|,
\end{equation}
because the equation $F(u_0) + a(0)u_0 -f_\delta=0$ is solvable.
If \eqref{deq47} holds, then
$$
h(0)e^{-\frac{t}{2}}\le \frac{1}{4}a(0)\|V_\delta(0)\|e^{-\frac{t}{2}},\qquad
t\ge 0.
$$
If $2b<c$, then  \eqref{ddeq47} implies $$e^{-\frac{t}{2}}a(0)\le a(t).$$
 Therefore,
\begin{equation}
\label{1eq20}
e^{-\frac{t}{2}}h(0)
\le \frac{1}{4} a(t)\|V_\delta(0)\|
\le \frac{1}{4} a(t)\|V_\delta(t)\|,\quad t\ge 0,
\end{equation}
where we have used the inequality $\|V_\delta(t)\|\le \|V_\delta(t')\|$ for
$t<t'$, established in Lemma~\ref{leminde} in Section~\ref{sec2}.
From \eqref{3eq23} and \eqref{1eq18}--\eqref{1eq20}, one gets
$$
C_1\delta^\zeta = \|F(u_\delta(t_\delta))-f_\delta\|\ge
\frac{1}{4}a(t_\delta)\|V_\delta(t_\delta)\|.
$$
Thus,
$$
\lim_{\delta\to0}a(t_\delta)\|V_\delta(t_\delta)\|\le
\lim_{\delta\to0}4C_1\delta^\zeta = 0.
$$
Since $\|V_\delta(t)\|$ increases (see Lemma~\ref{leminde}), the above
formula implies
$\lim_{\delta\to0}a(t_\delta)=0$. Since $0<a(t)\searrow 0$,
it follows that $\lim_{\delta \to 0}t_\delta=\infty$, i.e.,
\eqref{1eq15} holds.

It is now easy to finish the proof of the Theorem~\ref{mainthm}.

From the triangle inequality and inequalities \eqref{eq*} and
\eqref{rejected11} one obtains
\begin{equation}
\label{eqhic56}
\begin{split}
\|u_\delta(t_\delta) - y\| &\le \|u_{\delta}(t_\delta) - V_{\delta}(t_\delta)\| +
\|V(t_\delta) - V_\delta(t_\delta)\| + \|V(t_\delta) - y\|\\
&\le \frac{a(t_\delta)}{\lambda} + \frac{\delta}{a(t_\delta)} + \|V(t_\delta)-y\|.
\end{split}
\end{equation}
Note that $V(t_\delta) = V_{0,a(t_\delta)}$ (see equation \eqref{2eq2}).
From \eqref{1eq14}, \eqref{1eq15}, inequality \eqref{eqhic56} and Lemma~\ref{rejectedlem}, one obtains
\eqref{2eq4}.
Theorem~\ref{mainthm} is proved.
\end{proof}

\begin{remark}
\label{hehehihiha}
{\rm
The trajectory $u_\delta(t)$ remains in the ball
$B(u_0,R):=\{u: \|u-u_0\|<R\}$ for all $t\leq t_\delta$, where $R$ does
not depend on $\delta$ as $\delta\to 0$. Indeed,
estimates \eqref{eq*}, \eqref{2eq1} and \eqref{eqthieu} imply
\begin{equation}
%label{}
\begin{split}
\|u_\delta(t)-u_0\|&\le
\|u_\delta(t)-V_\delta(t)\|+\|V_\delta(t)\|+\|u_0\|\\
&\le
\frac{a(0)}{\lambda} +\frac{C\|y\|}{C-1} +\|u_0\|:=R,
\qquad \forall t\le t_\delta.
\end{split}
\end{equation}
Here we have used the fact that $t_\delta < t_0$ (see the proof of
Theorem~\ref{mainthm}).
Since one can choose $a(t)$ and $\lambda$ so that $\frac{a(0)}{\lambda}$
is uniformly bounded as $\delta\to 0$ and regardless of the growth of
$M_1$ (see Remark~\ref{rem8}) one concludes that $R$ can be chosen
independent of $\delta$ and $M_1$.
}
\end{remark}

%\section{Auxiliary results}
%\label{sec2}

\section{Numerical experiments}

\subsection{An experiment with an operator defined on $H=L^2[0,1]$}

Let us do a numerical experiment solving nonlinear equation \eqref{aeq1} with
\begin{equation}
\label{1eq41}
F(u):= B(u)+ \big{(}\arctan(u)\big{)}^3:=\int_0^1e^{-|x-y|}u(y)dy + \big{(}\arctan(u)\big{)}^3.
\end{equation}
%One can check that $u(x)\equiv 1$ solves  the equation $F(u)=f$.
%The operator $B$ is compact in $H=L^2[0,1]$. The operator $u\longmapsto u^3$
%is defined on a dense subset $D$ of of $L^2[0,1]$, for example, on $D:=C[0,1]$.
Since the function $u\to \arctan^3u$ is increasing on $\mathbb{R}$, one
has
\begin{equation}
\langle \big{(}\arctan(u)\big{)}^3 - \big{(}\arctan(v)\big{)}^3,u-v\rangle %= \int_0^1(\arctan^2(u)- \arctan^2(v))(u-v)dx
\ge 0,\qquad \forall\, u,v\in H.
\end{equation}
Moreover,
\begin{equation}
\label{eq77}
e^{-|x|} = \frac{1}{\pi}\int_{-\infty}^\infty \frac{e^{i\lambda x}}{1+\lambda^2} d\lambda.
\end{equation}
Therefore, $\langle B(u-v),u-v\rangle\ge0$, so
\begin{equation}
\langle F(u-v),u-v\rangle\ge0,\qquad \forall\, u,v\in H.
\end{equation}
Thus, $F$ is a monotone operator. Note that
$$
\langle \big{(}\arctan(u)\big{)}^3 - \big{(}\arctan(v)\big{)}^3,u-v\rangle %= \int_0^1(\arctan^2(u)- \arctan^2(v))(u-v)dx
= 0\quad \text{iff}\quad  u=v\quad a.e.
$$
Therefore, the operator $F$, defined in \eqref{1eq41}, is injective and equation \eqref{aeq1}, with this $F$, has at most one solution. 

The Fr\'{e}chet derivative of $F$ is
\begin{equation}
\label{eq44}
F'(u)w = \frac{3\big{(}\arctan(u)\big{)}^2}{1+u^2}w +
\int_0^1 e^{-|x-y|}w(y) dy.
\end{equation}
If $u(x)$ vanishes on a set of positive Lebesgue measure, then $F'(u)$
is not boundedly invertible.
If $u\in C[0,1]$ vanishes even at one point $x_0$, then $F'(u)$
is not boundedly invertible in $H$.

In numerical implementation of the DSM, one often discretizes the Cauchy 
problem \eqref{3eq12} 
and gets a system of ordinary differential equations (ODEs). 
Then, one can use numerical methods for solving ODEs to solve
the system of ordinary differential equations obtained from discretization. 
There are many numerical methods for solving ODEs (see, e.g.,
\cite{Hairer}).

In practice one does not have to compute $u_\delta(t_\delta)$
exactly but can use an approximation
to $u_\delta(t_\delta)$ as a stable solution to equation \eqref{aeq1}.
To calculate such an approximation, one can use, for example, the iterative 
scheme
\begin{equation}
\label{eq11-51}
\begin{split}
u_{n+1} &= u_n - (F'(u_n)+a_n I)^{-1}(F(u_n)+a_nu_n - f_\delta),\\
u_0 &= 0,
\end{split}
\end{equation}
and stop iterations at $n:=n_\delta$ such that the following inequality holds:
\begin{equation}
\label{eq53}
\|F(u_{n_\delta}) - f_\delta\| <C \delta^\gamma,\quad
\|F(u_{n}) - f_\delta\|\ge C\delta^\gamma,\quad n<n_\delta ,\quad C>1,\quad \gamma \in(0,1).
\end{equation}
The existence of the stopping time $n_\delta$ is proved  in \cite[p. 733]{R546}
and the choice $u_0=0$ is also justified in this paper.
%The drawback of the iterative scheme \eqref{eq11-51} compared to the DSM
%in this paper is that
%the solution $u_{n_\delta}$ may converge not to the minimal-norm
% solution to equation \eqref{aeq1} but to another solution to this
%equation, if this equation has many solutions.
Iterative scheme \eqref{eq11-51} and stopping rule \eqref{eq53} are used in
the numerical experiments. We proved in \cite[p. 733]{R546} that
$u_{n_\delta}$ converges to $u^*$, a solution of \eqref{aeq1}.
Since $F$ is injective as discussed above, we conclude that 
$u_{n_\delta}$ converges
to the unique solution of equation \eqref{aeq1} as $\delta$ tends to 0. 
%However, $u^*$ may be not the minimal-norm solution to \eqref{aeq1}.
The accuracy and stability are the key issues in solving the
Cauchy problem.
The iterative scheme \eqref{eq11-51} can be considered formally as the
explicit Euler's method
with the stepsize $h=1$ (see, e.g., \cite{Hairer}).
There might be  other iterative schemes which are more efficient
than scheme \eqref{eq11-51}, but this scheme is simple
and easy to implement.
% simplest method but it is
%not one of the most efficient methods for solving ODEs.

Integrals of the form
$\int_0^1 e^{-|x-y|}h(y)dy$ in \eqref{1eq41} and \eqref{eq44} are computed by
using
the trapezoidal rule. The noisy function used in the test is
$$
f_\delta(x) = f(x) + \kappa f_{noise}(x),\quad \kappa>0.
$$
The noise level $\delta$ and the relative noise level are defined by
the formulas
$$
\delta = \kappa\| f_{noise}\|,\quad \delta_{rel}:=\frac{\delta}{\|f\|}.
$$
In the test $\kappa$ is computed in such a way that the relative noise level
$\delta_{rel}$ equals some desired value, i.e.,
$$
\kappa = \frac{\delta}{\| f_{noise}\|}=\frac{\delta_{rel}\|f\|}{\| f_{noise}\|}.
$$
We have used the relative noise level as an input parameter in the test.

In all the figures the $x$-variable runs through the interval $[0,1]$,
and the graphs represent the numerical solutions
$u_{DSM}(x)$ and the exact solution $u_{exact}(x)$.

In the test we took $h=1$, $C =1.01$, and $\gamma = 0.99$.
The exact solution in the test is
\begin{equation}
\label{eqnewfunction}
u_e(x)=
\left \{
\begin{matrix} 0 &\text{if} \qquad \frac{1}{3}\le x\le \frac{2}{3},\\
1 & \text{otherwise},
\end{matrix}
\right .
\end{equation}
here $x\in[0,1]$, and the right-hand side
is $f=F(u_e)$.
As mentioned above, $F'(u)$ is not boundedly invertible 
in any neighborhood of $u_e$.

It is proved in \cite{R546} that one can take
$a_n=\frac{d}{1+n}$, and $d$ is
sufficiently large. However, in practice, if we choose $d$ too large,
then the method will use too many iterations before reaching the stopping
time $n_\delta$ in \eqref{eq53}. This means that the computation time will
be large in this case. Since
$$
\|F(V_\delta) - f_\delta\| = a(t)\|V_\delta\|,
$$
and $\|V_\delta(t_\delta) - u_\delta(t_\delta)\|=O(a(t_\delta))$,
we have
$$
C\delta^\gamma =\|F(u_\delta(t_\delta)) - f_\delta\|\leq
a(t_\delta)\|V_\delta\|+ O(a(t_\delta)),
$$
and we choose
$$
d = C_0\delta^\gamma,\qquad C_0>0.
$$
In the experiments our method works well with
$C_0\in[7,10]$.
In numerical experiments, we found out that the method diverged for smaller $C_0$.
In the
test we chose $a_n$ by the formula $a_n := C_0\frac{\delta^{0.99}}{n+1}$.
The number of nodal points, used in
computing integrals in \eqref{1eq41} and \eqref{eq44}, was $N=100$.
The accuracy of the solutions obtained
in the tests with $N=30$ and $N=50$ was
slightly less accurate than the one for $N=100$.

Numerical results for various values of $\delta_{rel}$ are presented in Table~\ref{table1}.
In this experiment,
the noise function $f_{noise}$ is a vector with random entries normally
distributed, with
mean value 0 and variance 1.
Table~\ref{table1} shows that the iterative scheme yields good numerical results.
\begin{table}[ht]
\caption{Results when $C_0=7$, $N=100$ and $u = u_e$.}
\label{table1}
\centering
\small
\begin{tabular}{|@{  }c@{\hspace{2mm}}
@{\hspace{2mm}}|c@{\hspace{2mm}}|c@{\hspace{2mm}}|c@{\hspace{2mm}}|c@{\hspace{2mm}}|
c@{\hspace{2mm}}|c@{\hspace{2mm}}|c@{\hspace{2mm}}r@{\hspace{2mm}}l@{}}
\hline
$\delta_{rel}$       &0.02   &0.01    &0.005    &0.003 &0.001\\
\hline
Number of iterations  &57     &57     &58   &58    &59\\
\hline
$\frac{\|u_{DSM} - u_{exact}\|}{\|u_{exact}\|}$  &0.1437   &0.1217    &0.0829   &0.0746    &0.0544\\
\hline
\end{tabular}
\end{table}

Figure~\ref{figmono} presents the numerical results when $N=100$ and
$C_0=7$ with $\delta_{rel}=0.01$ and $\delta_{rel} = 0.005$. The
numbers of iterations for $\delta=0.01$ and $\delta = 0.005$ were 57
and 58, respectively.
\begin{figure}[!h!tb]
\centerline{%
\includegraphics[scale=0.75]{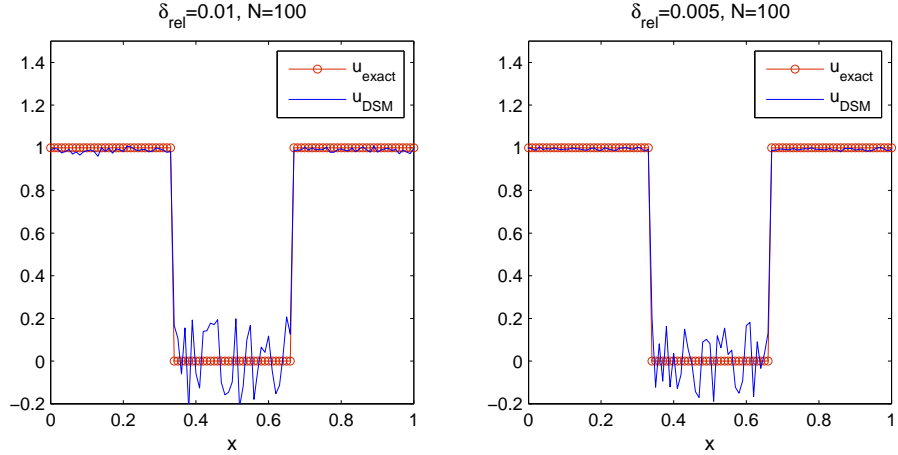}}
\caption{Plots solutions obtained by the DSM when $N = 100$, $\delta_{rel}=0.01$ (left)
and $\delta_{rel}=0.005$ (right).}
\label{figmono}
\end{figure}

Figure~\ref{figmono2} presents the numerical results
when $N=100$ and $C_0=7$ with $\delta=0.003$ and $\delta = 0.001$.
In these cases, it took 58 and 59 iterations to get the numerical
solutions for
$\delta_{rel}=0.003$ and $\delta_{rel} = 0.001$, respectively.

\begin{figure}[!h!tb]
\centerline{%
\includegraphics[scale=0.75]{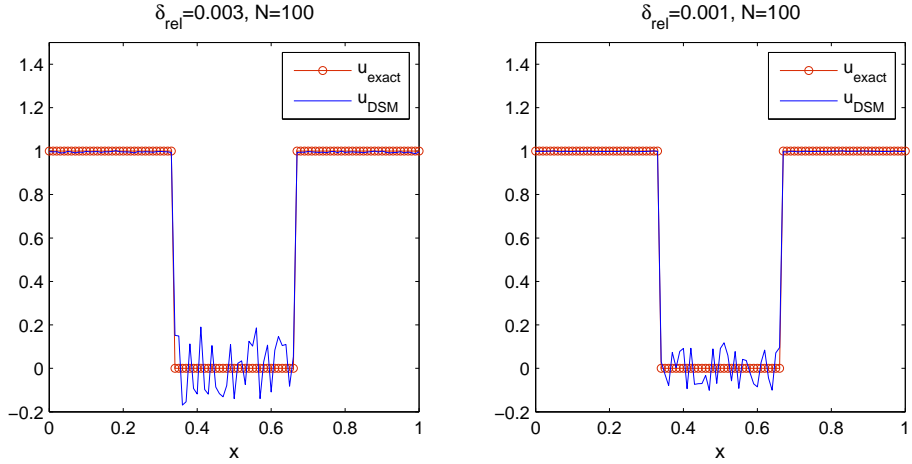}}
\caption{Plots solutions obtained by the DSM when $N = 100$, $\delta_{rel}=0.003$ (left)
and $\delta_{rel}=0.001$ (right).}
\label{figmono2}
\end{figure}

We also carried out numerical experiments with $u(x)\equiv 1, \, x\in [0,1]$, as the exact solution. 
Note that $F'(u)$ is boundedly invertible at this exact solution. 
However, in any arbitrarily small (in $L^2$ norm) neighborhood of this 
solution, 
there are infinitely many elements $u$ at which $F'(u)$ is not 
boundedly invertible, because, as we have pointed out earlier,
$F'(u)$ is not boundedly invertible if
$u(x)$ is continuous and vanishes at some point $x\in [0,1]$. 
In this case one cannot use the usual methods like 
Newton's method or the Newton-Kantorovich method. 
Numerical results for this experiment are presented in Table~\ref{table3}. 
\begin{table}[ht] 
\caption{Results when $C_0=4$, $N=50$ and $u(x)\equiv 1, \, x\in [0,1]$.}
\label{table3}
\centering
\small
\begin{tabular}{|@{  }c@{\hspace{2mm}}
@{\hspace{2mm}}|c@{\hspace{2mm}}|c@{\hspace{2mm}}|c@{\hspace{2mm}}|c@{\hspace{2mm}}|
c@{\hspace{2mm}}|c@{\hspace{2mm}}|c@{\hspace{2mm}}r@{\hspace{2mm}}l@{}} 
\hline
$\delta_{rel}$       &0.05   &0.03   &0.02    &0.01    &0.003 &0.001\\
\hline
Number of iterations &28	 &29     &28     &29   &29    &29\\
\hline 
$\frac{\|u_{DSM} - u_{exact}\|}{\|u_{exact}\|}$&0.0770	 &0.0411   &0.0314    &0.0146   &0.0046    &0.0015\\
\hline
\end{tabular}
\end{table}

From Table~\ref{table3} one concludes that the method works well in this experiment. 

\subsection{An experiment with an operator defined on a dense subset of $H=L^2[0,1]$}

Our second numerical experiment with the equation $F(u)=f$
deals with the operator $F$ which is not defined on all of $H=L^2[0,1]$,
but on a dense subset $D=C[0,1]$ of $H$:
\begin{equation}
F(u):= B(u)+u^3:=\int_0^1e^{-|x-y|}u(y)dy +u^3.% \quad f(x): = 3-e^{-x}-\frac{e^x}{e}.
\end{equation}
Therefore, the assumptions of Theorem~\ref{mainthm} are not satisfied.
Our goal is to show by this numerical example, that
numerically our method may work for an even wider class of problems than
that covered by Theorem~\ref{mainthm}.

%One can check that $u(x)\equiv 1$ solves  the equation $F(u)=f$.
The operator $B$ is compact in $H=L^2[0,1]$. The operator $u\longmapsto u^3$
is defined on a dense subset $D$ of $L^2[0,1]$, for example, on $D:=C[0,1]$.
If $u,v\in D$, then
\begin{equation}
\label{eqxy85}
\langle u^3-v^3,u-v\rangle = \int_0^1(u^3-v^3)(u-v)dx \ge 0.
\end{equation}
This and the inequality $\langle B(u-v),u-v\rangle\ge0$,
followed from equality \eqref{eq77}, imply
$$
\langle F(u-v),u-v\rangle\ge0,\qquad \forall u,v\in D.
$$
Note that the equal sign of inequality \eqref{eqxy85} happens iff $u=v$ a.e. in Lebesgue measure. 
Thus, $F$ is injective. Therefore, the element $u_{n_\delta}$ obtained from the iterative scheme 
\eqref{eq11-51} and the stopping rule \eqref{eq53} converges to the exact solution $u_e$
as $\delta$ goes to 0.

Note that $D$ does not contain subsets open in $H=L^2[0,1]$, i.e.,
it does not contain interior points of $H$.
This is a reflection of the fact that the operator $G(u)=u^3$
is unbounded on any open subset of $H$.
For example, in any ball $\|u\|\le C$, $C=const>0$,
where $\|u\|:=\|u\|_{L^2[0,1]}$, there is an element $u$ such that
$\|u^3\|=\infty$. As such an element one can take, for example,
$u(x)=c_1 x^{-b}$, $\frac{1}{3}<b<\frac{1}{2}$.
Here $c_1>0$ is a constant chosen so that
$\|u\|\leq C$.
The operator $u\longmapsto F(u)=G(u)+B(u)$ is maximal monotone on $D_F:=\{u:u\in H,\, F(u)\in H\}$ (see \cite[p. \#102]{D}),
so that equation \eqref{2eq2} is uniquely solvable for any $f_\delta\in H$.

The Fr\'{e}chet derivative of $F$ is
\begin{equation}
\label{eq442}
F'(u)w = 3u^2 w + \int_0^1 e^{-|x-y|}w(y) dy.
\end{equation}
If $u(x)$ vanishes on a set of positive Lebesgue measure, then $F'(u)$
is obviously not boundedly invertible.
If $u\in C[0,1]$ vanishes even at one point $x_0$, then $F'(u)$ is not boundedly invertible in $H$.

We also use the iterative scheme \eqref{eq11-51} with the stopping 
rule \eqref{eq53}.

We use the same exact solution $u_e$ as in \eqref{eqnewfunction}. The right-hand side $f$
is computed by $f=F(u_e)$. Note that $F'$ is not boundedly invertible in any neighborhood of $u_e$.

In experiments we found that our method works well with $C_0\in[1,4]$. Indeed, in the
test we chose $a_n$ by the formula $a_n := C_0\frac{\delta^{0.9}}{n+6}$.
The number of node points used in
computing integrals in \eqref{1eq41} and \eqref{eq44} was $N = 30$. In the
test, the accuracy of the solutions obtained when $N=30$, $N=50$ were
slightly less accurate than the one when $N = 100$.

Numerical results for various values of $\delta_{rel}$ are presented in Table~\ref{table2}.
In this experiment,
the noise function $f_{noise}$ is a vector with random entries normally distributed of
mean 0 and variance 1.
Table~\ref{table2} shows that the iterative scheme yields good numerical results.
\begin{table}[ht]
\caption{Results when $C_0=2$ and $N=100$.}
\label{table2}
\centering
\small
\begin{tabular}{|@{  }c@{\hspace{2mm}}
@{\hspace{2mm}}|c@{\hspace{2mm}}|c@{\hspace{2mm}}|c@{\hspace{2mm}}|c@{\hspace{2mm}}|
c@{\hspace{2mm}}|c@{\hspace{2mm}}|c@{\hspace{2mm}}r@{\hspace{2mm}}l@{}}
\hline
$\delta_{rel}$         &0.02   &0.01    &0.005    &0.003 &0.001\\
\hline
Number of iterations     &16     &17      &17   &17    &18\\
\hline
$\frac{\|u_{DSM} - u_{exact}\|}{\|u_{exact}\|}$  &0.1387   &0.1281    &0.0966   &0.0784    &0.0626\\
\hline
\end{tabular}
\end{table}

Figure~\ref{figmono3} presents the numerical results when
$f_{noise}(x) = \sin(3\pi x)$  for $\delta_{rel}=0.02$ and
$\delta_{rel} = 0.01$. The number of iterations when $C_0=2$ for
$\delta_{rel}=0.02$ and $\delta_{rel} = 0.01$ were 16 and 17,
respectively.

\begin{figure}[!h!tb]
\centerline{%
\includegraphics[scale=0.75]{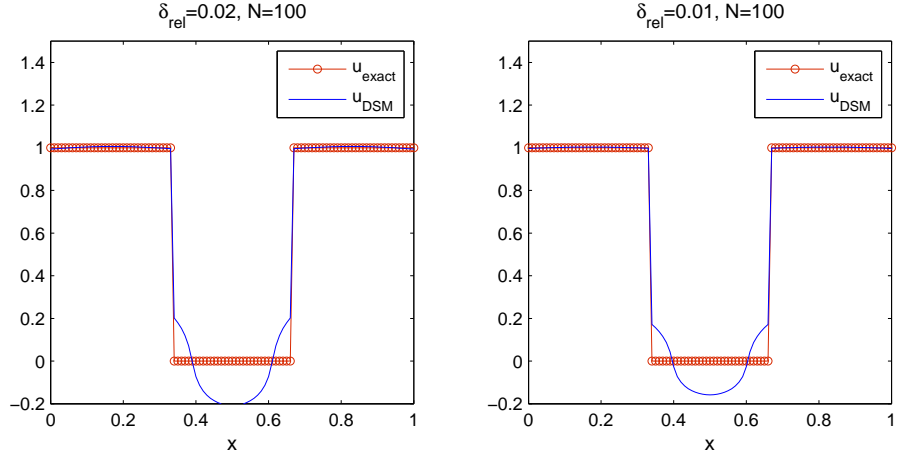}}
\caption{Plots solutions obtained by the DSM with $f_{noise}(x)=\sin(3\pi x)$ when $N = 100$, $\delta_{rel}=0.02$ (left)
and $\delta_{rel}=0.01$ (right).}
\label{figmono3}
\end{figure}

Figure~\ref{figmono4} presents the numerical results
when $f_{noise}(x) = \sin(3\pi x)$ with $\delta_{rel}=0.003$ and $\delta_{rel} = 0.001$.
We also used $C_0=2$.
In these cases, it took 17 and 18 iterations to give the numerical solutions for
$\delta_{rel}=0.003$ and $\delta_{rel} = 0.001$, respectively.

\begin{figure}[!h!tb]
\centerline{%
\includegraphics[scale=0.75]{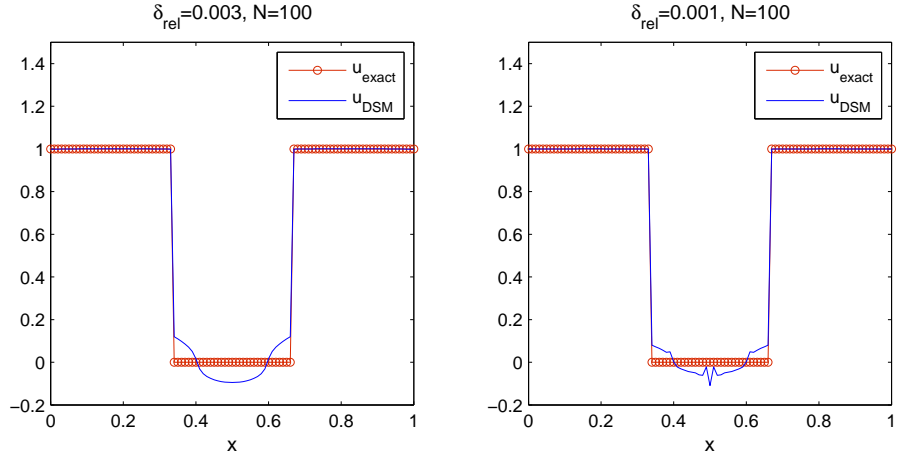}}
\caption{Plots solutions obtained by the DSM with $f_{noise}(x)=\sin(3\pi x)$ when $N = 100$, $\delta_{rel}=0.003$ (left)
and $\delta_{rel}=0.001$ (right).}
\label{figmono4}
\end{figure}

We have included the results of the  numerical experiments with 
$u(x)\equiv 1, \, x\in [0,1]$, as the exact solution. 
The operator $F'(u)$ is boundedly invertible in $L^2([0,1])$ at this exact 
solution.
However, in any arbitrarily small  $L^2$-neighborhood of this solution, 
there are infinitely many elements $u$ at which $F'(u)$ is not 
boundedly invertible
as was mentioned  above. 
Therefore, even in this case one cannot use the usual methods such as Newton's 
method or the Newton-Kantorovich method. 
Numerical results for this experiment are presented in Table~\ref{table4}.
\begin{table}[ht] 
\caption{Results when $C_0=1$, $N=30$ and $u(x)=1,\, x\in [0,1]$.}
\label{table4}
\centering
\small
\begin{tabular}{|@{  }c@{\hspace{2mm}}
@{\hspace{2mm}}|c@{\hspace{2mm}}|c@{\hspace{2mm}}|c@{\hspace{2mm}}|c@{\hspace{2mm}}|
c@{\hspace{2mm}}|c@{\hspace{2mm}}|c@{\hspace{2mm}}r@{\hspace{2mm}}l@{}} 
\hline
$\delta_{rel}$       &0.05   &0.03   &0.02    &0.01    &0.003 &0.001\\
\hline
Number of iterations &7	 &8     &8      &9   &10    &10\\
\hline 
$\frac{\|u_{DSM} - u_{exact}\|}{\|u_{exact}\|}$&0.0436	 &0.0245   &0.0172    &0.0092   &0.0026    &0.0009\\
\hline
\end{tabular}
\end{table}

From the numerical experiments we can conclude that the method works well in this experiment.
Note that the function $F$ used in this experiment is not defined on the whole space $H=L^2[0,1]$
but defined on a dense subset $D=C[0,1]$ of $H$.

\end{document}